\documentclass[]{amsart}

\usepackage[english]{babel}
\usepackage[utf8]{inputenc}
\usepackage[T1]{fontenc}
\usepackage{microtype}
\usepackage{paralist}
\usepackage{amsmath, amsfonts, amssymb, amsthm}
\usepackage{color}
\usepackage[normalem]{ulem}
\usepackage{tikz}

\newcommand{\CC}{\mathbb{C}}

\newcommand{\NN}{\mathbb{N}}

\newcommand{\ZZ}{\mathbb{Z}}

\newcommand{\Hc}{\mathcal{H}}

\newcommand{\Oc}{\mathcal{O}}

\newcommand{\gfr}{\mathfrak{g}}
\newcommand{\set}[1]{\left\{ #1 \right\}}
\newcommand{\setb}[1]{\left( #1 \right)}
\newcommand{\abs}[1]{\left| #1 \right|}

\newcommand{\bino}[2]{\begin{pmatrix} #1 \\ #2 \end{pmatrix}}
\newcommand{\caux}{C_{\mathrm{aux}}}

\newtheorem{mymasterthm}{notForUse}
\theoremstyle{definition}

\theoremstyle{plain}
\newtheorem{mylemma}[mymasterthm]{Lemma}
\newtheorem{mythm}[mymasterthm]{Theorem}
\newtheorem{mycorol}[mymasterthm]{Corollary}
\newtheorem{myprop}[mymasterthm]{Proposition}

\allowdisplaybreaks

\title[Size of a linear combination of two l.r.s.]{On the size of a linear combination of two linear recurrence sequences over function fields}
\makeatletter
\@namedef{subjclassname@2020}{\textup{2020} Mathematics Subject Classification}
\makeatother
\subjclass[2020]{11B37}
\keywords{Linear recurrence sequence, growth}

\author[S. Heintze]{Sebastian Heintze}
\address{Sebastian Heintze\newline
	\indent Graz University of Technology\newline
	\indent Institute of Analysis and Number Theory\newline
	\indent Steyrergasse 30/II \newline
	\indent A-8010 Graz, Austria}
\email{heintze@math.tugraz.at}

\thanks{Supported by Austrian Science Fund (FWF) under project I4406}

\begin{document}
	
	\maketitle
	
	
	\begin{abstract}
		Let $ G_n $ and $ H_m $ be two non-degenerate linear recurrence sequences defined over a function field $ F $ in one variable over $ \CC $, and let $ \mu $ be a valuation on $ F $.
		We prove that under suitable conditions there are effectively computable constants $ c_1 $ and $ C' $ such that the bound
		\begin{equation*}
			\mu(G_n - H_m) \leq \mu(G_n) + C'
		\end{equation*}
		holds for $ \max \setb{n,m} > c_1 $.
	\end{abstract}
	
	\section{Introduction}
	
	Linear recurrence sequences are studied by many authors in the past and until now.
	Here, by a linear recurrence sequence we mean a polynomial-exponential function, from the set $ \NN_0 $ of non-negative integers into a given field $ F $, of the form
	\begin{equation*}
		G_n = a_1(n) \alpha_1^n + \cdots + a_d(n) \alpha_d^n,
	\end{equation*}
	where the $ \alpha_i $ are called the characteristic roots of the linear recurrence sequence and the coefficients $ a_i(n) $ are polynomials in $ n $.
	It is well known that such a sequence satisfies are linear recurring formula.
	We say that the seuquence $ (G_n)_{n \in \NN_0} $ is defined over the field $ F $ if all characteristic roots $ \alpha_i $ as well as all coefficients of all polynomials $ a_i(n) $ belong to $ F $.
	The recurrence sequence is called non-degenerate if no ratio of two distinct characteristic roots $ \alpha_i/\alpha_j $ for $ i \neq j $ is a root of unity in the case that $ F $ is a number field, or if no ratio of two distinct characteristic roots $ \alpha_i/\alpha_j $ for $ i \neq j $ is contained in the field of constants when $ F $ is a function field in one variable over $ \CC $, respectively.
	
	In \cite{fuchs-heintze-p3} the author together with Fuchs gave a bound on the size of the $ n $-th element of such a linear recurrence sequence defined over a function field, see Proposition~\ref{prop:boundgrowth} below.
	They also provide a proof for a well known bound on the growth of $ G_n $ in the case that $ F $ is a number field in the appendix of \cite{fuchs-heintze-p3}.
	
	Recently, Peth\H{o} \cite{petho-} considered the size of the difference of two linear recurrence sequences over number fields.
	More precisely, it is proven that for two recurrences $ A_n $ and $ B_m $, taking only integer values, under some technical conditions ($ A_n $ has a dominant root, i.e.\ there is a unique characteristic root $ \alpha $ with maximal absolute value, $ B_m $ has a pair of conjugate complex dominating characteristic roots, and some further assumptions) the bound
	\begin{equation*}
		\abs{A_n - B_m} > \abs{A_n}^{1-(c_0 \log^2 n)/n}
	\end{equation*}
	holds for $ (n,m) \in \NN_0^2 $ with $ \max \setb{n,m} > c_1 $, where $ c_0,c_1 $ are effectively computable constants.
	
	The purpose of the present paper is to find and prove a suitable similar bound in the setting of function fields in one variable over the field of complex numbers.
	
	\section{Notation and results}
	
	Throughout this paper we denote by $ F $ a function field in one variable over $ \CC $ and by $ \gfr $ the genus of $ F $.
	For the convenience of the reader we will give a short wrap-up of the notion of valuations that can e.g.\ also be found in \cite{fuchs-heintze-p3,fuchs-heintze-p8}:
	For $ c \in \CC $ and $ f(x) \in \CC(x)^* $, where $ \CC(x) $ is the rational function field over $ \CC $, we denote by $ \nu_c(f) $ the unique integer such that $ f(x) = (x-c)^{\nu_c(f)} p(x) / q(x) $ with $ p(x),q(x) \in \CC[x] $ such that $ p(c)q(c) \neq 0 $.
	Further we write $ \nu_{\infty}(f) = \deg q - \deg p $ if $ f(x) = p(x) / q(x) $.
	Additionally, we set $ \nu(0) = \infty $ for each $ \nu $ from above.
	These functions $ \nu : \CC(x) \rightarrow \ZZ \cup \set{\infty} $ are up to equivalence all valuations in $ \CC(x) $.
	If $ \nu_c(f) > 0 $, then $ c $ is called a zero of $ f $, and if $ \nu_c(f) < 0 $, then $ c $ is called a pole of $ f $, where $ c \in \CC \cup \set{\infty} $.
	For a finite extension $ F $ of $ \CC(x) $ each valuation in $ \CC(x) $ can be extended to no more than $ [F : \CC(x)] $ valuations in $ F $.
	This again gives up to equivalence all valuations in $ F $.
	Both, in $ \CC(x) $ as well as in $ F $ the sum-formula
	\begin{equation*}
		\sum_{\nu} \nu(f) = 0
	\end{equation*}
	holds for each nonzero $ f $, where the sum is taken over all valuations in the considered function field.
	Moreover, valuations have the properties $ \nu(fg) = \nu(f) + \nu(g) $ and $ \nu(f+g) \geq \min \setb{\nu(f), \nu(g)} $ for all $ f,g \in F $.
	For more information about valuations we refer to \cite{stichtenoth-1993}.
	
	For a finite set $ S $ of valuations on $ F $, we denote by $ \Oc_S^* $ the set of $ S $-units in $ F $, i.e.\ the set
	\begin{equation*}
		\Oc_S^* = \set{f \in F^* : \nu(f) = 0 \text{ for all } \nu \notin S}.
	\end{equation*}
	Lastly, we call two elements $ \alpha, \beta \in F $ multiplicatively independent if $ \alpha^r \beta^s \in \CC $ for $ r,s \in \ZZ $ implies that $ r = s = 0 $.
	
	Our first result is now the following theorem which states that there cannot be much cancelation in the expression $ a G_n - b H_m $ if both indices are large:
	\begin{mythm}
		\label{thm:bothlarge}
		Let $ G_n = a_1(n) \alpha_1^n + \cdots + a_d(n) \alpha_d^n $ and $ H_m = b_1(m) \beta_1^m + \cdots + b_t(m) \beta_t^m $ be two non-degenerate linear recurrence sequences defined over $ F $.
		Assume that $ \alpha_1 \notin \CC $, and that for any $ j \in \set{1,\ldots,t} $ the pair $ (\alpha_1,\beta_j) $ is multiplicatively independent.
		Furthermore, let $ \mu $ be a valuation on $ F $ such that $ \mu(\alpha_1) \leq \mu(\alpha_i) $ for $ i \in \set{1,\ldots,d} $.
		Fix $ a,b \in F^* $.
		Then there exist effectively computable constants $ c_0 $ and $ C $, independent of $ n $ and $ m $, such that for $ \min \setb{n,m} > c_0 $ we have
		\begin{equation*}
			\mu(a G_n - b H_m) \leq \mu(G_n) + C.
		\end{equation*}
	\end{mythm}
	
	The non-degeneracy condition already implies that there is at most one characteristic root in each of the two linear recurrences which is constant.
	If we require all characteristic roots to be non-constant, then we can prove a little bit more:
	\begin{mythm}
		\label{thm:onelarge}
		Let $ G_n = a_1(n) \alpha_1^n + \cdots + a_d(n) \alpha_d^n $ and $ H_m = b_1(m) \beta_1^m + \cdots + b_t(m) \beta_t^m $ be two non-degenerate linear recurrence sequences defined over $ F $.
		Assume that no $ \alpha_i $ as well as no $ \beta_j $ is contained in $ \CC $, and that for any $ j \in \set{1,\ldots,t} $ the pair $ (\alpha_1,\beta_j) $ is multiplicatively independent.
		Furthermore, let $ \mu $ be a valuation on $ F $ such that $ \mu(\alpha_1) \leq \mu(\alpha_i) $ for $ i \in \set{1,\ldots,d} $.
		Fix $ a,b \in F^* $.
		Then there exist effectively computable constants $ c_1 $ and $ C' $, independent of $ n $ and $ m $, such that for $ \max \setb{n,m} > c_1 $ we have
		\begin{equation*}
			\mu(a G_n - b H_m) \leq \mu(G_n) + C'.
		\end{equation*}
	\end{mythm}
	
	In the case $ \mu(aG_n) \neq \mu(bH_m) $ the inequality directly follows from the strict triangle inequality.
	Thus the power of the above theorems concentrates on the case $ \mu(aG_n) = \mu(bH_m) $.
	There they give a nontrivial upper bound, whereas the trivial lower bound in the case $ \mu(aG_n) = \mu(bH_m) $ is
	\begin{equation*}
		\mu(a G_n - b H_m) \geq \min \setb{\mu(aG_n), \mu(bH_m)} = \mu(aG_n) = \mu(G_n) + \mu(a).
	\end{equation*}
	Rephrased in words, our theorems state that for large indices the recurrence $ H_m $ cannot cancel out too much from $ G_n $ if at least one ``size-determining'' root $ \alpha_1 $ is independent of the roots of $ H_m $.
	
	The assumption that $ \alpha_1 $ is multiplicatively independent of each characteristic root of the second recurrence sequence is needed to avoid situations like $ H_m := G_{2m} $, where $ G_n - H_m $ is zero for $ n = 2m $ arbitrary large, and thus the statement of the theorems cannot hold.
	That things are different if the two considered linear recurrence sequences are too similar, can also be seen in the results of other authors, see e.g.\ \cite{fuchs-petho-2005}.
	Let us mention that, as in Corollary 4 in \cite{fuchs-petho-2005}, we can deduce here that under the assumptions of Theorem \ref{thm:onelarge} the solutions $ (n,m) $ to $ aG_n = bH_m $ are bounded effectively from above.
	
	From Theorem \ref{thm:bothlarge} to Theorem \ref{thm:onelarge} we extended the area, in which the bound for the valuation holds, from $ \min \setb{n,m} > c_0 $ to $ \max \setb{n,m} > c_1 $ to the cost of a little bit stronger assumptions.
	The restriction $ \max \setb{n,m} > c_1 $ cannot be removed completely.
	Indeed, there may be sporadic solutions to $ a G_n - b H_m = 0 $ whence $ \mu(a G_n - b H_m) = \infty $ is possible for small indices.
	
	To illustrate the result, we formulate the following corollary which immediately follows from Theorem \ref{thm:bothlarge} by choosing $ \mu = \nu_{\infty} $ for the function field $ \CC(x) $.
	An analogous corollary can be formulated for Theorem \ref{thm:onelarge}.
	\begin{mycorol}
		\label{cor:polycase}
		Let $ G_n = a_1(n) \alpha_1^n + \cdots + a_d(n) \alpha_d^n $ and $ H_m = b_1(m) \beta_1^m + \cdots + b_t(m) \beta_t^m $ be two non-degenerate linear recurrence sequences of polynomials in $ \CC[x] $ where all the characteristic roots are polynomials as well.
		Assume that $ \alpha_1 \notin \CC $, and that for any $ j \in \set{1,\ldots,t} $ the pair $ (\alpha_1,\beta_j) $ is multiplicatively independent.
		Furthermore, assume that $ \deg \alpha_1 \geq \deg \alpha_i $ for $ i \in \set{1,\ldots,d} $.
		Fix nonzero $ a,b \in \CC[x] $.
		Then there exist effectively computable constants $ c_0 $ and $ C $, independent of $ n $ and $ m $, such that for $ \min \setb{n,m} > c_0 $ we have
		\begin{equation*}
			\deg (a G_n - b H_m) \geq \deg G_n - C.
		\end{equation*}
	\end{mycorol}
	
	\section{Preliminaries}
	
	In the next section we will make use of height functions in function fields. Let us therefore define the height of an element $ f \in F^* $ by
	\begin{equation*}
		\Hc(f) := - \sum_{\nu} \min \setb{0, \nu(f)} = \sum_{\nu} \max \setb{0, \nu(f)}
	\end{equation*}
	where the sum is taken over all valuations on the function field $ F / \CC $. Additionally we define $ \Hc(0) = \infty $.
	This height function satisfies some basic properties that are listed in the lemma below which is proven in \cite{fuchs-karolus-kreso-2019}:
	\begin{mylemma}
		\label{lemma:heightproperties}
		Denote as above by $ \Hc $ the height on $ F/\CC $. Then for $ f,g \in F^* $ the following properties hold:
		\begin{enumerate}[a)]
			\item $ \Hc(f) \geq 0 $ and $ \Hc(f) = \Hc(1/f) $,
			\item $ \Hc(f) - \Hc(g) \leq \Hc(f+g) \leq \Hc(f) + \Hc(g) $,
			\item $ \Hc(f) - \Hc(g) \leq \Hc(fg) \leq \Hc(f) + \Hc(g) $,
			\item $ \Hc(f^n) = \abs{n} \cdot \Hc(f) $,
			\item $ \Hc(f) = 0 \iff f \in \CC^* $,
			\item $ \Hc(A(f)) = \deg A \cdot \Hc(f) $ for any $ A \in \CC[T] \setminus \set{0} $.
		\end{enumerate}
	\end{mylemma}
	
	Moreover, the following result due to Brownawell and Masser will be used when proving our statements. It is an immediate consequence of Theorem B in \cite{brownawell-masser-1986}:
	\begin{myprop}[Brownawell-Masser]
		\label{prop:brownawellmasser}
		Let $ F/\CC $ be a function field in one variable of genus $ \gfr $. Moreover, for a finite set $ S $ of valuations, let $ u_1,\ldots,u_k $ be $ S $-units and
		\begin{equation*}
			1 + u_1 + \cdots + u_k = 0,
		\end{equation*}
		where no proper subsum of the left hand side vanishes. Then we have
		\begin{equation*}
			\max_{i=1,\ldots,k} \Hc(u_i) \leq \bino{k}{2} \left( \abs{S} + \max \setb{0, 2\gfr-2} \right).
		\end{equation*}
	\end{myprop}
	
	Furthermore, we will use the following function field analogue of the Schmidt subspace theorem. A proof can be found in \cite{zannier-2008}:
	\begin{myprop}[Zannier]
		\label{prop:functionfieldsubspace}
		Let $ F/\CC $ be a function field in one variable, of genus $ \gfr $, let $ \varphi_1, \ldots, \varphi_n \in F $ be linearly independent over $ \CC $ and let $ r \in \set{0,1, \ldots, n} $. Let $ S $ be a finite set of places of $ F $ containing all the poles of $ \varphi_1, \ldots, \varphi_n $ and all the zeros of $ \varphi_1, \ldots, \varphi_r $. Put $ \sigma = \sum_{i=1}^{n} \varphi_i $. Then
		\begin{equation*}
			\sum_{\nu \in S} \left( \nu(\sigma) - \min_{i=1, \ldots, n} \nu(\varphi_i) \right) \leq \bino{n}{2} (\abs{S} + 2\gfr - 2) + \sum_{i=r+1}^{n}  \Hc(\varphi_i).
		\end{equation*}
	\end{myprop}
	
	In addition, the next proposition will be applied in our proofs. It is proven as Theorem 1 in \cite{fuchs-heintze-p3} and we state it here in a combined version with the paragraph immediately before Theorem 1 in \cite{fuchs-heintze-p3}:
	\begin{myprop}
		\label{prop:boundgrowth}
		Let $ (G_n)_{n=0}^{\infty} $ be a non-degenerate linear recurrence sequence taking values in $ F $ with power sum representation $ G_n = a_1(n) \alpha_1^n + \cdots + a_t(n) \alpha_t^n $. Let $ L $ be the splitting field of the characteristic polynomial of that sequence, i.e.\ $ L = F(\alpha_1, \ldots, \alpha_t) $. Moreover, let $ \mu $ be a valuation on $ L $. Then there are effectively computable constants $ C^+ $ and $ C^- $, independent of $ n $, such that for every sufficiently large $ n $ the inequality
		\begin{equation*}
			C^- + n \cdot \min_{j=1,\ldots,t} \mu(\alpha_j) \leq \mu(G_n) \leq C^+ + n \cdot \min_{j=1,\ldots,t} \mu(\alpha_j)
		\end{equation*}
		holds.
	\end{myprop}
	Note that an inspection of the proof of the last proposition shows that it is possible to calculate a (admittedly rather complicated) bound $ N_0 $ such that ``sufficiently large $ n $'' can be replaced by $ n \geq N_0 $.
	
	Last but not least, we will need the following small lemma about multiplicatively independent elements, which is proven in \cite{fuchs-heintze-p8}:
	\begin{mylemma}
		\label{lemma:quotofindep}
		Let $ \gamma, \delta \in F \setminus \CC $ be multiplicatively independent and $ n,m \in \NN $. Assume that
		\begin{equation*}
			\Hc \left( \frac{\gamma^n}{\delta^m} \right) \leq L.
		\end{equation*}
		Then there exists an effectively computable constant $ L' $, depending only on $ \gamma, \delta, \gfr $ and $ L $, such that
		\begin{equation*}
			\max \setb{n,m} \leq L'.
		\end{equation*}
	\end{mylemma}
	
	\section{Proofs}
	
	We have prepared all auxiliary results needed for proving our theorems.
	Thus we can start with the proof of our first theorem.
	\begin{proof}[Proof of Theorem \ref{thm:bothlarge}]
		First note that $ a G_n $ is again a non-degenerate linear recurrence sequence with the same characteristic roots as $ G_n $ and that $ \mu(aG_n) = \mu(a) + \mu(G_n) $.
		The analogue holds for $ b H_m $.
		So, without loss of generality, we may assume that $ a = b = 1 $.
		
		Let us rewrite the linear recurrence sequences in a more suitable manner.
		With
		\begin{equation*}
			a_i(n) = \sum_{k=0}^{e_i} a_{ik} n^k
		\end{equation*}
		we can write
		\begin{equation}
			\label{eq:Gnfirst}
			G_n = \sum_{i=1}^{d} a_i(n) \alpha_i^n = \sum_{i=1}^{d} \sum_{k=0}^{e_i} a_{ik} n^k \alpha_i^n.
		\end{equation}
		Now fix for each $ i \in \set{1,\ldots,d} $ a maximal $ \CC $-linear independent subset $ \set{\pi_{i1}, \ldots, \pi_{ik_i}} $ of $ \set{a_{i0}, \ldots, a_{ie_i}} $.
		Using these elements, we can write \eqref{eq:Gnfirst} as
		\begin{equation*}
			G_n = \sum_{i=1}^{d} \sum_{g=1}^{k_i} P_{ig}(n) \pi_{ig} \alpha_i^n
		\end{equation*}
		for polynomials $ P_{ig}(n) \in \CC[n] $.
		Analogously, we get
		\begin{equation*}
			H_m = \sum_{j=1}^{t} \sum_{h=1}^{\ell_j} Q_{jh}(m) \psi_{jh} \beta_j^m
		\end{equation*}
		where $ Q_{jh}(m) \in \CC[m] $ are polynomials and $ \set{\psi_{j1}, \ldots, \psi_{j\ell_j}} $ is linearly independent over $ \CC $ for any $ j \in \set{1,\ldots,t} $.
		Together these representations yield
		\begin{equation}
			\label{eq:GnHm}
			G_n - H_m = \sum_{i=1}^{d} \sum_{g=1}^{k_i} P_{ig}(n) \pi_{ig} \alpha_i^n - \sum_{j=1}^{t} \sum_{h=1}^{\ell_j} Q_{jh}(m) \psi_{jh} \beta_j^m.
		\end{equation}
		
		In order to be able to apply Proposition \ref{prop:functionfieldsubspace} we would need the summands in \eqref{eq:GnHm} to be linearly independent over $ \CC $.
		Therefore we will check this in the sequel and make changes where necessary.
		The procedure for doing so is as follows:
		We assume that we have given an arbitrary but fixed pair $ (n,m) $ of indices and, considering several cases, deduce that then either $ \min \setb{n,m} \leq c_0 $, which falls out of the scope of the statement where we only say something for $ \min \setb{n,m} > c_0 $, or a related (but in general slightly modified) sum to \eqref{eq:GnHm} consists of $ \CC $-linear independent summands.
		During this procedure, the bound $ c_0 $ will be updated several (but only finitely many) times without changing its label, i.e.\ it is always denoted by $ c_0 $.
		As an initial value we choose $ c_0 $ large enough such that
		\begin{equation*}
			\prod_{i=1}^{d} \prod_{g=1}^{k_i} P_{ig}(n) \cdot \prod_{j=1}^{t} \prod_{h=1}^{\ell_j} Q_{jh}(m)
		\end{equation*}
		is nonzero whenever $ \min \setb{n,m} > c_0 $.
		
		Now suppose that the summands in \eqref{eq:GnHm} are linearly dependent over $ \CC $.
		Then we have complex numbers $ \lambda_{ig}, \gamma_{jh} \in \CC $, not all zero, such that
		\begin{equation}
			\label{eq:lindep}
			\sum_{i=1}^{d} \sum_{g=1}^{k_i} \lambda_{ig} P_{ig}(n) \pi_{ig} \alpha_i^n + \sum_{j=1}^{t} \sum_{h=1}^{\ell_j} \gamma_{jh} Q_{jh}(m) \psi_{jh} \beta_j^m = 0.
		\end{equation}
		Note that the $ \lambda_{ig} $ and $ \gamma_{jh} $ may depend on $ (n,m) $ which we assume as fixed for this consideration.
		Now we consider a minimal vanishing subsum of \eqref{eq:lindep}, i.e.\ no subsubsum of this subsum vanishes.
		In particular, all $ \lambda_{ig} $ and $ \gamma_{jh} $ appearing in this minimal vanishing subsum are nonzero.
		Moreover, we fix a finite set $ S $ of valuations such that all $ \alpha_i, \beta_j, \pi_{ig} $ and $ \psi_{jh} $ are $ S $-units, and such that $ \mu \in S $, and define the constant
		\begin{equation*}
			\caux := \bino{\sum_{i=1}^{d} k_i + \sum_{j=1}^{t} \ell_j}{2} \left( \abs{S} + \max \setb{0, 2 \gfr - 2} \right).
		\end{equation*}
		Both, $ S $ and $ \caux $ are independent of $ n $ and $ m $.
		We distinguish between six cases:
		
		\textbf{Case 1:}
		The minimal vanishing subsum contains only summands with the same factor $ \alpha_i^n $.
		Recalling that $ \set{\pi_{i1}, \ldots, \pi_{ik_i}} $ is linearly independent over $ \CC $, we see that this case is not possible.
		
		\textbf{Case 2:}
		The minimal vanishing subsum contains only summands with the same factor $ \beta_j^m $.
		Recalling that $ \set{\psi_{j1}, \ldots, \psi_{j\ell_j}} $ is linearly independent over $ \CC $, we see that this case is also not possible.
		
		\textbf{Case 3:}
		The minimal vanishing subsum contains summands with the factors $ \alpha_i^n $ and $ \alpha_j^n $, respectively, where $ i \neq j $.
		Dividing the minimal vanishing subsum by a summand containing the factor $ \alpha_j^n $ and then applying Proposition \ref{prop:brownawellmasser} (note that all summands are $ S $-units since $ \lambda_{ig}, P_{ig}(n), \gamma_{jh}, Q_{jh}(m) \in \CC $) yields
		\begin{equation*}
			\Hc \left( \frac{\lambda_{ig} P_{ig}(n) \pi_{ig} \alpha_i^n}{\lambda_{jg'} P_{jg'}(n) \pi_{jg'} \alpha_j^n} \right) \leq \caux
		\end{equation*}
		for some indices $ g,g' $.
		By Lemma \ref{lemma:heightproperties}, this implies
		\begin{equation*}
			n \cdot \Hc \left( \frac{\alpha_i}{\alpha_j} \right) = \Hc \left( \frac{\alpha_i^n}{\alpha_j^n} \right) \leq \caux + \Hc \left( \frac{\pi_{ig}}{\pi_{jg'}} \right)
		\end{equation*}
		and, since $ G_n $ is non-degenerate, further
		\begin{equation}
			\label{eq:case3}
			n \leq \frac{\caux + \max_{i,j,g,g'} \Hc \left( \dfrac{\pi_{ig}}{\pi_{jg'}} \right)}{\min_{i \neq j} \Hc \left( \dfrac{\alpha_i}{\alpha_j} \right)}.
		\end{equation}
		The upper bound in \eqref{eq:case3} is independent of $ n $ and $ m $ and thus, for an updated $ c_0 $ we get $ \min \setb{n,m} \leq n \leq c_0 $.
		
		\textbf{Case 4:}
		The minimal vanishing subsum contains summands with the factors $ \beta_i^m $ and $ \beta_j^m $, respectively, where $ i \neq j $.
		This case is handled completely analogously to the previous one.
		
		\textbf{Case 5:}
		The minimal vanishing subsum contains summands with the factors $ \alpha_1^n $ and $ \beta_j^m $, respectively.
		Dividing the minimal vanishing subsum by a summand containing the factor $ \beta_j^m $ and then applying Proposition \ref{prop:brownawellmasser} yields
		\begin{equation*}
			\Hc \left( \frac{\lambda_{1g} P_{1g}(n) \pi_{1g} \alpha_1^n}{\gamma_{jh} Q_{jh}(m) \psi_{jh} \beta_j^m} \right) \leq \caux
		\end{equation*}
		for some indices $ g,h $.
		By Lemma \ref{lemma:heightproperties}, this implies
		\begin{equation*}
			\Hc \left( \frac{\alpha_1^n}{\beta_j^m} \right) \leq \caux + \Hc \left( \frac{\pi_{1g}}{\psi_{jh}} \right).
		\end{equation*}
		From this we get either, again by Lemma \ref{lemma:heightproperties},
		\begin{equation*}
			n \leq \frac{\caux + \max_{j,g,h} \Hc \left( \dfrac{\pi_{1g}}{\psi_{jh}} \right)}{\Hc (\alpha_1)}
		\end{equation*}
		if $ \beta_j \in \CC $, or, by Lemma \ref{lemma:quotofindep},
		\begin{equation*}
			\max \setb{n,m} \leq L'
		\end{equation*}
		if $ \beta_j \notin \CC $.
		In both subcases, the upper bound is independent of $ n $ and $ m $, and thus we get $ \min \setb{n,m} \leq c_0 $, for an updated $ c_0 $.
		
		\textbf{Case 6:}
		The minimal vanishing subsum contains summands with the factors $ \alpha_i^n $ and $ \beta_j^m $, respectively, where $ i \neq 1 $.
		In particular, we may assume that no summand with a factor $ \alpha_1^n $ is contained.
		Then we can dissolve the minimal vanishing subsum after one of the appearing terms of the shape $ Q_{jh}(m) \psi_{jh} \beta_j^m $, i.e.\ express this term by a $ \CC $-linear combination of the remaining terms in this subsum.
		Now we insert this expression for $ Q_{jh}(m) \psi_{jh} \beta_j^m $ into \eqref{eq:GnHm}, summarize terms which differ only by a constant factor, and get recurrences $ G_n' $ as well as $ H_m' $ with the following properties:
		We have $ G_n - H_m = G_n' - H_m' $ for the considered pair $ (n,m) $, all expressions of the shape $ \pi_{ig} \alpha_i^n $ or $ \psi_{jh} \beta_j^m $ appearing in $ G_n' - H_m' $ also appear in $ G_n - H_m $ (in general with different coefficients in $ \CC $), no summand containing $ \pi_{1g} \alpha_1^n $ got lost, and $ G_n' - H_m' $ has less summands than $ G_n - H_m $.
		
		Next we check whether the summands in $ G_n' - H_m' $ are linearly independent over $ \CC $.
		If not, then we do the same as we have done above with $ G_n - H_m $.
		Observe that we are automatically in Case 6 again since we are only interested in $ \min \setb{n,m} > c_0 $.
		Here we perform the same reduction process to get $ G_n'' - H_m'' $.
		As in each reduction process the number of summands reduces, this iteration ends after finitely many steps, and after renumbering terms (note that $ \alpha_1 $ stays $ \alpha_1 $ since terms containing $ \alpha_1 $ can not be removed during the reduction process) we get
		\begin{equation}
			\label{eq:star}
			G_n - H_m = G_n^* - H_m^* := \sum_{i=1}^{d^*} \sum_{g=1}^{k_i^*} P_{ig}^*(n) \pi_{ig} \alpha_i^n - \sum_{j=1}^{t^*} \sum_{h=1}^{\ell_j^*} Q_{jh}^*(m) \psi_{jh} \beta_j^m.
		\end{equation}
		Note that $ d^* \geq 1 $ and $ k_1^* \geq 1 $, i.e.\ $ \alpha_1 $ appears on the right hand side.
		The summands in the expression on the right hand side of equation \eqref{eq:star} are now linearly independent over $ \CC $ because we only consider $ \min \setb{n,m} > c_0 $ and no further reduction steps were possible.
		Nevertheless, which summands from $ G_n - H_m $ still appear in $ G_n^* - H_m^* $ may depend on the considered pair $ (n,m) $.
		However, this will not be a problem in the sequel since the number of summands is bounded uniformly (cf.\ our definition of $ \caux $).
		
		At this point we are now able to apply Proposition \ref{prop:functionfieldsubspace}.
		By our choice of $ S $, each summand of the right hand side of equation \eqref{eq:star} is an $ S $-unit.
		Put
		\begin{equation*}
			r_1 := \sum_{i=1}^{d^*} k_i^* \qquad \text{as well as} \qquad r_2 := \sum_{j=1}^{t^*} \ell_j^*
		\end{equation*}
		and set
		\begin{equation*}
			(\varphi_1, \ldots, \varphi_{r_1}) := \left( P_{ig}^*(n) \pi_{ig} \alpha_i^n \right)_{i,g}
		\end{equation*}
		for an arbitrary ordering of the summands of $ G_n^* $ as well as
		\begin{equation*}
			(\varphi_{r_1+1}, \ldots, \varphi_{r_1+r_2}) := \left( -Q_{jh}^*(m) \psi_{jh} \beta_j^m \right)_{j,h}
		\end{equation*}
		for an arbitrary ordering of the summands of $ H_m^* $.
		With this notation, Proposition~\ref{prop:functionfieldsubspace} implies
		\begin{equation}
			\label{eq:valbound}
			\sum_{\nu \in S} \left( \nu(G_n^* - H_m^*) - \min_{z=1,\ldots,r_1+r_2} \nu(\varphi_z) \right) \leq \caux.
		\end{equation}
		Since $ G_n - H_m = G_n^* - H_m^* $, since each summand in the sum on the left hand side of inequality \eqref{eq:valbound} is non-negative, and since $ \mu \in S $, we get
		\begin{equation*}
			\mu(G_n-H_m) - \min_{z=1,\ldots,r_1+r_2} \mu(\varphi_z) \leq \caux.
		\end{equation*}
		From this we infer
		\begin{align*}
			\mu(G_n-H_m) &\leq \caux + \min_{z=1,\ldots,r_1+r_2} \mu(\varphi_z) \\
			&\leq \caux + \mu \left( P_{11}^*(n) \pi_{11} \alpha_1^n \right) \\
			&= \caux + \mu(\pi_{11}) + n \cdot \mu(\alpha_1) \\
			&= \caux + \mu(\pi_{11}) + n \cdot \min_{i=1, \ldots, d} \mu(\alpha_i) \\
			&\leq \caux + \mu(\pi_{11}) + \mu(G_n) - C^- \\
			&= \mu(G_n) + C,
		\end{align*}
		where in the second to last line we have used Proposition \ref{prop:boundgrowth} and $ c_0 $ becomes updated for the last time.
		This proves the theorem.
	\end{proof}
	
	The assumptions in our second theorem contain all assumptions from Theorem \ref{thm:bothlarge}.
	So it is not surprising that the proof of it builds on Theorem \ref{thm:bothlarge}.
	\begin{proof}[Proof of Theorem \ref{thm:onelarge}]
		By Theorem \ref{thm:bothlarge}, there exist constants $ c_0 $ and $ C $ such that for $ \min \setb{n,m} > c_0 $ we have
		\begin{equation*}
			\mu(a G_n - b H_m) \leq \mu(G_n) + C.
		\end{equation*}
		It remains to consider the case when one index is small.
		
		Therefore let, firstly, $ m \leq c_0 $ be fixed.
		Then $ H_m $ is fixed as well.
		Since there are only finitely many such cases, we can perform the following for each of this cases and write $ H_{(m)} $ for $ H_m $ in the calculation to emphasize that we consider only a fixed value for $ m $ each time.
		Put $ \alpha_{d+1} := 1 $ and consider the linear recurrence sequence
		\begin{equation*}
			\widetilde{G_n} := a G_n - b H_{(m)} \alpha_{d+1}^n = a G_n - b H_{(m)}.
		\end{equation*}
		As $ G_n $ is non-degenerate and has no constant characteristic root, $ \widetilde{G_n} $ is also non-degenerate.
		Thus Proposition \ref{prop:boundgrowth} yields
		\begin{align*}
			\mu \left( a G_n - b H_{(m)} \right) &= \mu \left( \widetilde{G_n} \right) \\
			&\leq C_{(m)}^+ + n \cdot \min_{i=1, \ldots, d+1} \mu (\alpha_i) \\
			&\leq C_{(m)}^+ + n \cdot \min_{i=1, \ldots, d} \mu (\alpha_i) \\
			&\leq C_{(m)}^+ + \mu(G_n) - C^- \\
			&= \mu(G_n) + C_{(m)}
		\end{align*}
		for $ n > c_{1,(m)} $.
		
		Consider now the second possibility, namely that $ n \leq c_0 $ is fixed.
		Then $ G_n $ is fixed as well.
		Since there are only finitely many such cases, we can perform the following for each of this cases and write $ G_{(n)} $ for $ G_n $ in the calculation to emphasize that we consider only a fixed value for $ n $ each time.
		Put $ \beta_{t+1} := 1 $ and consider the linear recurrence sequence
		\begin{equation*}
			\widetilde{H_m} := a G_{(n)} \beta_{t+1}^m - b H_m = a G_{(n)} - b H_m.
		\end{equation*}
		As $ H_m $ is non-degenerate and has no constant characteristic root, $ \widetilde{H_m} $ is also non-degenerate.
		So Proposition \ref{prop:boundgrowth} yields
		\begin{align*}
			\mu \left( a G_{(n)} - b H_m \right) &= \mu \left( \widetilde{H_m} \right) \\
			&\leq C_{(n)}^+ + m \cdot \min_{j=1, \ldots, t+1} \mu (\beta_j) \\
			&\leq C_{(n)}^+ + m \cdot \mu (\beta_{t+1}) \\
			&= C_{(n)}^+ \\
			&= \mu \left( G_{(n)} \right) + C_{(n)}
		\end{align*}
		for $ m > c_{1,(n)} $.
		
		Finally, we put
		\begin{equation*}
			c_1 := \max \setb{c_0, \max_{m \leq c_0} c_{1,(m)}, \max_{n \leq c_0} c_{1,(n)}}
		\end{equation*}
		and
		\begin{equation*}
			C' := \max \setb{C, \max_{m \leq c_0} C_{(m)}, \max_{n \leq c_0} C_{(n)}}.
		\end{equation*}
		For these constants, it holds that
		\begin{equation*}
			\mu(a G_n - b H_m) \leq \mu(G_n) + C'
		\end{equation*}
		whenever $ \max \setb{n,m} > c_1 $, and the theorem is proven.
	\end{proof}

\end{document}